\documentclass[a4paper, 11pt]{amsart}
\usepackage{longtable, stackrel, multicol, fancybox, stackrel, tcolorbox, mathtools}
\usepackage{bm,amsfonts,amsmath,amssymb,mathrsfs,bbm,amsthm} 
\usepackage{graphicx, color,hyperref,eurosym, eucal,palatino}
\usepackage{scrextend}
\changefontsizes{10.1pt}

\definecolor{darkblue}{rgb}{0.2,0.2,0.6}
\definecolor{superdarkblue}{rgb}{0.2,0.2,0.3}
\hypersetup{colorlinks=true, linkcolor=darkblue,citecolor=darkblue, urlcolor = superdarkblue}
\definecolor{citegreen}{rgb}{0.2,0.2,0.6}
\definecolor{green2}{rgb}{0.,0.6,0.2}

\newcommand\nb{\nabla}
\newcommand{\beq}{\begin{equation} \begin{split}}
\newcommand{\eeq}{\end{split} \end{equation}}
\newcommand\Sg{\Sigma}
\newcommand\Omg{\Omega}

\newcommand\fru{u_\Omg}
\newcommand\avg[1]{\langle #1\rangle}

\makeatletter
\def\section{\@startsection{section}{1}\z@{.9\linespacing\@plus\linespacing}%
	{.7\linespacing} {\fontsize{13}{14}\selectfont\bfseries\centering}}
\def\paragraph{\@startsection{paragraph}{4}%
	\z@{0.3em}{-.5em}%
	{$\bullet$ \ \normalfont\itshape}}

\@addtoreset{equation}{section}
\makeatother

\renewcommand\and{\qquad\text{and}\qquad}

\newcommand\sm{\setminus}

\newcommand\dl{\delta}

\newcommand{\comm}[1]{}

\def\bm1{\mathbbm{1}}
\def\G{\Gamma}

\def\s{\sigma}

\def\p{\partial}

\def\omg{\omega}

\def\arr{\rightarrow}

\def\lm{\lambda}

\def\s{\sigma}

\def\p{\partial}

\def\kp{\kappa}

\def\dd{{\mathsf{d}}}

\def\omg{\omega}

\newcounter{counter_a}
\newenvironment{myenum}{\begin{list}{{\rm(\roman{counter_a})}}%
{\usecounter{counter_a}
\setlength{\itemsep}{1.ex}\setlength{\topsep}{0.8ex}
\setlength{\leftmargin}{5ex}\setlength{\labelwidth}{5ex}}}{\end{list}}

\usepackage[latin1]{inputenc}
\usepackage[T1]{fontenc}

\newcommand{\eg}{{\it e.g.}\,}
\newcommand{\ie}{{\it i.e.}\,}
\newcommand{\cf}{{\it cf.}\,}

\numberwithin{figure}{section}
\numberwithin{equation}{section}
\theoremstyle{plain}
\newtheorem*{thm*}{Theorem}
\newtheorem{thm}{Theorem}[section]
\newtheorem{hyp}[thm]{Hypothesis}
\newtheorem{lem}[thm]{Lemma}
\newtheorem{prop}[thm]{Proposition}
\newtheorem{example}[thm]{Example}

\newtheorem{cor}[thm]{Corollary}

\newtheorem{dfn}[thm]{Definition}
\theoremstyle{remark}
\newtheorem{remark}[thm]{Remark}
\theoremstyle{plain}

\newcommand{\supp}{\mathrm{supp}\,}
\newcommand{\beu}{\begin{equation*}}
\newcommand{\eeu}{\end{equation*}}
\newcommand{\besu}{\begin{equation*}
\begin{aligned}}
\newcommand{\eesu}{\end{aligned}
\end{equation*}}
\newcommand{\bes}{\begin{equation}
\begin{aligned}}
\newcommand{\ees}{\end{aligned}
\end{equation}}

\newcommand\cA{\mathcal A}
\newcommand\cB{\mathcal B}

\newcommand\cF{\mathcal F}

\newcommand\cS{\mathcal S}

\newcommand\fra{\mathfrak a}

\newcommand\frh{\mathfrak h}

\newcommand\eps{\varepsilon}

\newcommand\ov{\overline}
\newcommand\wt{\widetilde}

\newcommand\void[1]{}

\def\ov{\overline}
\def\eps{\varepsilon}

      \def\dC{{\mathbb C}}

      \def\dR{{\mathbb R}}

   \def\dZ{{\mathbb Z}}

\def\sfA{{\mathsf A}}

\def\cA{{\mathcal A}}   \def\cB{{\mathcal B}}   \def\cC{{\mathcal C}}
      \def\cF{{\mathcal F}}

\def\cS{{\mathcal S}}      
\def\cV{{\mathcal V}}

\newcommand{\dom}{\mathrm{dom}\,}

\newcommand{\vl}[1]{{\color{darkblue} #1}}

\title[]{Trace Hardy inequality for the Euclidean space with a cut and its applications}

\author[M. Dauge]{Monique Dauge}
\address{Institut de recherche math\'ematique de Rennes (IRMAR -- UMR CNRS 6625)
Universit\'e de Rennes 1,
263 avenue du G\'en\'eral Leclerc,
35042 Rennes CEDEX,
France
}
\author[M. Jex]{Michal Jex} 
\address{
Department of Physics\\
Faculty of Nuclear Sciences and Physical Engineering, Czech Technical University in Prague, B\v rehov\'a 7, 11519 Prague, Czech Republic\\
E-mail: {michal.jex@fjfi.cvut.cz }
}
\author[V. Lotoreichik]{Vladimir Lotoreichik}
\address{
Department of Theoretical Physics\\
Nuclear Physics Institute, Czech Academy of Sciences, 
25068 \v{R}e\v{z}, Czech Republic\\
E-mail: {lotoreichik@ujf.cas.cz }
}

\begin{document}

\begin{abstract}
	We obtain a trace Hardy inequality for the Euclidean space with a bounded cut $\Sigma\subset\dR^d$, $d \ge 2$. In this novel geometric setting, the Hardy-type inequality non-typically holds also for $d = 2$. 
	The respective Hardy weight is given
	in terms of the geodesic distance to the boundary of $\Sg$. We provide its applications
   to the heat equation on $\dR^d$ with an insulating cut at $\Sigma$ and to the Schr\"odinger operator with a $\delta'$-interaction supported on 
	$\Sigma$. We also obtain generalizations of this trace Hardy inequality for a class of unbounded cuts.
\end{abstract}
\maketitle	
\section{Introduction}
%
The classical Hardy inequality in the Euclidean space $\dR^d$, $d \ge 3$, can be stated as follows
\begin{equation}\label{eq:Hardy}
\int_{\dR^d}|\nabla u|^2\dd x \ge 
\frac{(d-2)^2}{4}\int_{\dR^d}\frac{|u|^2}{|x|^2} \dd x, 
\qquad \forall\, u\in H^1(\dR^d).
\end{equation}
This inequality and some of its generalizations 
can be found in \eg~\cite[\S 5.3]{D95}.
They provide a technical tool in various studies of elliptic partial differential operators. 
Probably, the most famous of them is Kato's proof of self-adjointness and semi-boundedness 
for many-body quantum Hamiltonians~\cite{K51}; 
\cf the discussion in the review~\cite{S17}. 
From this perspective, the physically observed stability of the hydrogen atom 
is just a simple consequence of the inequality~\eqref{eq:Hardy}.
We also point out usefulness of~\eqref{eq:Hardy} in the proof of self-adjointness and semi-boundedness 
for a class of Schr\"{o}dinger operators with potentials having decoupled singularities~\cite{GMNT16}.
Further applications of the Hardy inequality concern criticality properties of 
the Schr\"o\-dinger operators~\cite{DFP14} and large time decay of the heat 
semigroups~\cite{VZ00}, the latter being intimately connected with the notion 
of transiency of the Brownian motion in stochastic analysis; \cf~\cite{KK14} 
and the references therein. 

Recently, considerable interest has been attracted by various trace versions of 
the Hardy inequality~\cite{FMT13, N11, T15, T16, vH16}. 
The progenitor of all the trace Hardy inequalities
and also the most illustrative of them is frequently attributed to T.~Kato and claims that
\begin{equation}\label{eq:Hardy_trace}
\int_{\dR^d_+}|\nabla u|^2\dd x 
\ge 
2\left( \frac{\G(\frac{d}{4})}{\G(\frac{d-2}{4})}\right )^2 
\int_{ \p\dR^d_+ } \frac{|u|^2}{|x'|} \dd x',	\qquad \forall \, u\in H^1(\dR^d_+),
\end{equation}
where $\dR^d_+ = \{(x',x_d)\in\dR^d\colon x'\in\dR^{d-1}, x_d > 0\}$, $d \ge 3$,
denotes the upper half-space in $\dR^d$ and $\G(\cdot)$ stands for the Euler 
$\G$-function. The inequality~\eqref{eq:Hardy_trace} 
can be derived from the Hardy inequality
for the relativistic Schr\"odinger operator~\cite{H77,Kato} and the Sobolev trace theorem
for the half-space. It is applicable \eg~in the 
study of Schr\"odinger operators with surface potentials~\cite{ES88, FL08}.

The main result of the present paper,
formulated in Theorem~\ref{thm:bnd}, is a trace Hardy inequality in a novel geometric setting of the Euclidean space with a cut 
across an interface, the latter being a bounded, non-closed hypersurface $\Sg\subset\dR^d$, $d \ge 2$.
By saying that $\Sg$ is non-closed, we essentially
mean that $\dR^d\setminus\ov\Sg$ is connected. In this setting, we replace the ordinary trace by its jump across $\Sg$. Under an additional Ahlfors-David-type regularity assumption,
the weight on $\Sigma$ involved in the inequality is estimated in terms of the geodesic distance to its boundary $\p\Sigma$.
Trace Hardy inequalities for the jump hold also in the two-dimensional case, thus being substantially different from the classical 
trace Hardy inequalities.

The trace Hardy inequality for $\dR^d\setminus\ov\Sg$ implies an
integral upper bound on the large time decay for 
the weighted $L^2$-norm of the temperature-jump 
across $\Sg$ for the solutions of the heat equation with the (insulating) Neumann boundary condition on
the cut $\Sg$ being imposed. The singularity of the
respective weight on $\p\Sigma$
reflects the intuition that the temperature levels off quicker in a neighbourhood
of $\p\Sigma$.

An analogous trace Hardy inequality is also proven
in Theorem~\ref{thm:cone} for a class of unbounded cuts. Namely, we consider a cut across the hypersurface $\Sg\subset\G:=\dR^{d-1}\times\{0\}$, $d \ge 3$, which is a (hyper)cone. Some generalizations for cuts with non-trivial topology such as M\"{o}bius strips are also discussed.

Considerations of the present paper are largely inspired by the 
spectral analysis of the Schr\"o\-dinger
operator with the $\dl'$-interaction supported on a non-closed curve carried out by two of us in~\cite{JL16}. The investigation of
Schr\"odinger operators with $\dl'$-interactions supported on hypersurfaces became a topic of permanent interest -- see,
\eg, \cite{BGLL15, BEL14, BLL13AHP, EJ13, EKh15, EKh18,
	L18, MPS16}. 
Such Hamiltonians appear, \eg, in the study of photonic crystals~\cite{FK96a,FK96b} and arise in the
asymptotic limit for a class of structured thin Neumann obstacles~\cite{DFZ18, Kh70}.
The  obtained trace Hardy inequality for a bounded cut implies absence of the negative discrete spectrum for a weak attractive $\delta'$-interaction supported on a bounded non-closed surface $\Sigma$.
This is in shear contrast to attractive $\delta'$-interaction supported on a closed bounded hypersurface, which always induces negative discrete spectrum (see~\cite[Thm. 4.4]{BEL14}).


\subsection*{Organization of the paper}
In Section~\ref{sec:mainres} we describe
the setting in more detail and formulate our main results. Section~\ref{sec:Sobolev} provides
several statements on Sobolev spaces that are used in the proofs.  Section~\ref{sec:bndcut} is devoted to the proof of the trace Hardy inequality for the jump
of the trace across a cut.
Finally, in Section~\ref{appl}
we apply the obtained inequality to the
estimation of the large-time
behaviour for the heat equation on $\dR^d\setminus\ov\Sg$ and to
the Schr\"odinger operator with a $\delta'$-interaction
supported on $\Sg$. In Section~\ref{sec:unb} we consider extensions to a class of unbounded cuts and to topologically non-trivial cuts.
%

\section{Setting of the problem and the main result}\label{sec:mainres}

Throughout the paper, we consider a cut $\Sigma$ in $\dR^d$ satisfying the following assumptions.
\begin{hyp}
\label{hyp:1}
We consider a relatively open subset $\Sigma\subset\Gamma$ of a bounded, orientable, and closed Lipschitz hypersurface\footnote{This means that $\Gamma$ is a Lipschitz manifold without boundary, \ie $\Gamma$ can be covered by a finite collection of local maps in which $\Gamma$ coincides with the graph of a Lipschitz function.} $\Gamma$. We implicitly assume that $\Gamma$ is the boundary of a bounded (Lipschitz) domain $\Omega_+\subset\dR^d$. We suppose moreover that $\dR^d\setminus\ov\Sigma$ is connected (thus $\Gamma\setminus\ov\Sigma$ is not empty).
\end{hyp}

Let $\partial\Sigma=\overline\Sigma\setminus\Sigma$ denote the boundary of $\Sigma$ in $\Gamma$. 
%
%
For any $x \in \Sg$, we abbreviate by $\rho_\Sg(x)$ the geodesic distance 
between $x$ and $\p\Sg$, measured in the induced Riemannian metric of $\Sg$ (\cf~\cite{CP91})
and by $\wt\rho_\Sg(x)$ we denote the respective Euclidean distance in $\dR^d$.

We denote by $\cB_r(x)\subset\dR^d$ the open 
ball of radius $r > 0$ centred at $x\in\dR^d$ and by $\cS\subset\dR^d$ the unit sphere. 
Recall also that $\cF\subset\dR^d$ is called a \emph{$(d-1)$-set} if there exist $r_\star >0$ and $c_-, c_+ > 0$ such that
for any $r \in (0,r_\star)$ one has
\begin{equation}\label{eq:(d-1)}
c_- r^{d-1}\le \Lambda_{d-1}\big(\cB_r(x)\cap \cF\big) \le 
c_+ r^{d-1},\qquad \forall\, x\in \cF,
\end{equation}
where $\Lambda_{d-1}(\cdot)$ stands for the $(d-1)$-dimensional Hausdorff measure. 

Sobolev spaces on domains with cuts are thoroughly investigated in the 
monographs~\cite{Gr85, D88}; see also~\cite{CHM17}.
Roughly speaking, the Sobolev space 
$H^1(\dR^d\setminus\ov\Sg)$ can be defined as a subspace of $H^1(\Omg_+)\oplus H^1(\dR^d\setminus\ov{\Omg_+})$
in which the traces of functions agree on $\G\setminus\ov\Sg$; \cf~Section~\ref{sec:Sobolev} for 
details. 

For any $u\in H^1(\dR^d\setminus\ov\Sg)$
the traces $u|_{\Sg_\pm}$ onto two faces 
$\Sg_\pm$ of $\Sg$ are well-defined functions in $L^2(\Sg)$. 
These functions need 
not be the same and, thus, 
the jump of the trace $[u]_\Sg := u|_{\Sg_+} - u|_{\Sg_-}$ is a well-defined 
and non-trivial function in $L^2(\Sg)$. 
%

%

Our trace Hardy inequality holds for functions in the Sobolev space $H^1(\dR^d\setminus\ov\Sg)$. If the complement
$\G\sm\Sg$ is a $(d-1)$-set, then the Hardy weight
in this inequality can be controlled by the inverse of the geodesic distance $\rho_\Sg(\cdot)$. 
%
\begin{thm}\label{thm:bnd}
	Let $\Sg\subset\dR^d$ be a bounded non-closed Lipschitz hypersurface satisfying  Hypothesis~\ref{hyp:1}. Then there exists a constant $C = C(\Sg) > 0$ such that 
	\[
	\int_{\dR^d}|\nb u(x)|^2\dd x 
	\ge 
	C\int_\Sg
	w(x) |[u]_\Sg(x)|^2\dd\s(x),
	\qquad 
	\forall\, u\in H^1(\dR^d\setminus\ov\Sg),
	\]
	where the weight $w\colon\Sg\arr\dR_+$ is given by
	\begin{equation}\label{eq:HardyWeight}
	w(x) := \int_{\G\setminus\ov\Sg} \frac{\dd\s(y)}{|x-y|^d}
	\end{equation}
	If, in addition, $\G\sm\Sg$ is a $(d-1)$-set, then
	there exists $c >0$ such that
	$w(x) \ge c\wt\rho_\Sg(x)^{-1}\ge c\rho_\Sg(x)^{-1}$.
\end{thm}

The  condition that $\G\sm\Sg$ is a $(d-1)$-set calls for the following remarks and examples:

\begin{remark}
(i)  The upper bound $\Lambda_{d-1}\big(\cB_r(x)\cap \cF\big) \le c_+ r^{d-1}$ is always satisfied for $\cF=\G\sm\Sg$ since $\Gamma$ itself is a $(d-1)$-set.
\smallskip

(ii) If the uniform lower bound $c_- r^{d-1}\le \Lambda_{d-1}\big(\cB_r(x)\cap \cF\big)$ is satisfied with $\cF=\G\sm\Sg$ for any $x\in\partial\Sigma$, this implies a similar bound for all $x\in\Gamma\setminus\Sigma$.
\smallskip

(iii) If $\Gamma'$ is another closed Lipschitz hypersurface containing $\Sigma$, and if $\G\sm\Sg$ is a $(d-1)$-set, then $\G'\sm\Sg$ is a $(d-1)$-set too.
\end{remark}

\begin{example}
{\rm (i)} If $\Sigma$ is (moreover) a Lipschitz hypersurface with Lipschitz boundary (its boundary $\p\Sg$  
can be viewed as a $(d-2)$-dimensional Lipschitz manifold) so $\G\sm\Sg$ is a $(d-1)$-set.
\smallskip

{\rm (ii)} In the space dimension $d\! =\! 2$, $\G\sm\Sg$ is a $1$-set without any extra assumptions on $\Sg$.

{\rm (iii)} If $\Sigma$ has an inward cusp at a point $x_0$, then $\G\sm\Sg$ has an outward cusp at the same point. The lower bound $c_- r^{d-1}\le \Lambda_{d-1}\big(\cB_r(x)\cap \cF\big)$ is not satisfied with $\cF=\G\sm\Sg$ at this point $x=x_0$.
\end{example} 
The existence of a weight singular on $\p\Sigma$
is reminiscent of the Hardy
inequality for the Dirichlet Laplacian on a bounded
planar regular Euclidean domain~\cite[Thms. 5.3.5 and 5.3.6]{D95}.
We also remark that the inequality in Theorem~\ref{thm:bnd} can be generalized to the setting of
a cut in a Riemannian manifold. However, we do not pursue this goal here.

An inequality analogous to the trace Hardy inequality in Theorem~\ref{thm:bnd} is proven
in Theorem~\ref{thm:cone} for a class of unbounded cuts. These cuts are contained in the hyperplane $\G :=\dR^{d-1}\times\{0\}$, $d \ge 3$ and a cut from this class is defined as
\[
	\Sg = \left\{x\in\dR^d\sm\{0\}\colon \frac{x}{|x|}\in\hat\Sg\right\},
\]
where $\hat\Sg$ is an open set in $\hat\G := \G\cap\cS$ such that $\hat\G\sm\hat\Sg$ is a $(d-2)$-set. The proof of Theorem~\ref{thm:cone}
relies on splitting the Euclidean space $\dR^d$ by dyadic spherical shells, inside which the method
used in the proof of Theorem~\ref{thm:bnd} is applied.
Similar localization technique can be used
to show generalizations of Theorem~\ref{thm:bnd} for non-orientable cuts such as M\"{o}bius strips in $\dR^3$; see Example~\ref{ex:moebius}.

\section{The first order Sobolev space on a domain with a cut}
\label{sec:Sobolev}
In this section we recall several properties of and notions related to the Sobolev space on a domain with a cut.
Let the hypersurface $\Sg\subset\dR^d$ satisfy Hypothesis~\ref{hyp:1}. Recall that we implicitly
assume existence of a bounded Lipschitz domain $\Omg_+\subset\dR^d$ with the
boundary $\G := \p\Omg_+$ such that $\Sg$ is a relatively open subset of $\G$. 
From the topological point of view, $\G$ and $\Sg$ possess two faces $\G_\pm$ and $\Sg_\pm$, respectively.
Additionally, we assume that $\Sg$ is non-closed or in other terms the domain $\dR^d\setminus\ov\Sg$ is assumed to be connected.
Further, let $\Omg\subseteq\dR^d$ be an open set such that $\ov{\Omg_+}\subset\Omg$ and that $\Omg_- := \Omg\setminus\ov{\Omg_+}$ is a possibly unbounded Lipschitz domain of the type
described  in~\cite[\S VI.3]{St}. 
The case $\Omg = \dR^d$ and $\Omg_- = \dR^d\setminus\ov{\Omg_+}$
is for us the most typical.

For a function $u\in L^2(\Omg)$ we adopt the notation
$u_\pm := u|_{\Omg_\pm}$.
The Sobolev space $H^1(\Omg\setminus\ov\G) = H^1(\Omg_+) \oplus H^1(\Omg_-)$ is introduced in the conventional
way. For a function $u = u_+ \oplus u_-\in H^1(\Omg\setminus\ov\G)$ we define
$\nabla u := \nabla u_+\oplus \nabla u_-$.
The norm on $H^1(\Omg\setminus\ov\G)$ is given by
\begin{equation}\label{key}
	\|u\|_1^2 = \|u\|^2_{H^1(\Omg\setminus\ov\G)} := 
		\|u\|^2_{L^2(\Omg)} + \|\nabla u\|^2_{L^2(\Omg;\dC^d)}.
\end{equation}
Next, we define the sub-space of 
$H^1(\Omg\setminus\ov\G)$ 
\[
	\cC_{\rm c}^\infty(\Omg\setminus\ov\Sg) := 
	\big\{
	u = u_+\oplus u_-  \in 
	\cC^\infty_0(\ov{\Omg_+})
	\oplus\cC^\infty_0(\ov{\Omg_-})
	\colon \supp([u]_{\G}  ) \subset\ov{\Sg}
	\big\},
\]
where $[u]_{\G} := u_+|_{\G} - u_-|_{\G}$.
\begin{dfn}\label{def:SobolevCut}
	The first-order $L^2$-based Sobolev space 
	on the domain $\Omg\setminus\ov\Sg$
	is defined as the closure of $\cC_{\rm c}^\infty(\Omg\setminus\ov\Sg)$ in $H^1(\Omg\setminus\ov\G)$ with respect to its norm
	\begin{equation}\label{eq:Sobolev}
		H^1(\Omg\setminus\ov\Sg) = \ov{\cC_{\rm c}^\infty(\Omg\setminus\ov\Sg)}^{\|\cdot\|_1}.
	\end{equation}
\end{dfn}
\begin{remark}
	The above definition of the Sobolev space $H^1(\Omg\setminus\ov\Sg)$ is not intrinsic.
	The Sobolev space $H^1(\Omg\setminus\ov\Sg)$ can be
	alternatively intrinsically defined as the space of all square-integrable functions on $\Omg$ such that their distributional gradients on $\Omg\setminus\ov\Sg$ are also square-integrable. The intrinsic definition is for technical reasons less convenient. The equivalence between the two definitions can be shown via standard argument based on integration by parts.
\end{remark}
Note that it is easy to see that
the inclusions 
$H^1(\Omg)\subsetneq H^1(\Omg\setminus\ov\Sg)$ and
$H^1(\Omg\setminus\ov\Sg) \subsetneq H^1(\Omg_+)\oplus H^1(\Omg_-)$ are proper.

Under the assumptions imposed, the trace maps
\[
	H^1(\Omg\setminus\ov\Sg)\ni u\mapsto u|_{\Sg_\pm} \in L^2(\Sg)
\]
onto two different faces $\Sg_\pm$ of $\Sg$ are well-defined, continuous, linear operators.
These traces need not coincide for a generic element
of $H^1(\Omg\setminus\ov\Sg)$. 

The jump of the trace $[u]_{\G} := u_+|_{\G} - u_-|_{\G}$ extends by continuity to the Sobolev
space $H^1(\Omg\setminus\ov\G)$. It is convenient
to introduce the notation $[u]_\Sg$ and $[u]_{\G\setminus\ov\Sg}$
for the restrictions of $[u]_\G$ onto $\Sg$ and $\G\setminus\ov\Sg$, respectively.
%
\begin{prop}\label{prop:Sobolev}
	For any $u\in H^1(\Omg\setminus\ov\Sg)$
	one has $[u]_{\G\setminus\ov\Sg} = 0$.
\end{prop}
\begin{proof}
	Let $u\in H^1(\Omg\setminus\ov\Sg)$ be fixed and let $(u_n)_n$ be a sequence in $\cC^\infty_{\rm c}(\Omg\setminus\ov\Sg)$ such that $\|u_n - u\|_1\arr 0$
	as $n\arr\infty$.
	Then by continuity we get
	\[
		[u]_{\G\setminus\ov\Sg} = \lim_{n\arr\infty}[u_n]_{\G\setminus\ov\Sg} = 0.
		\qedhere
	\]
\end{proof}	
%
The fractional Sobolev space $H^{1/2}(\G)$
is defined as in~\cite[\S 1.3.3]{Gr85}  via local charts (see also~\cite[Chap. 3]{McL}). 
There are several ways 
to define the norm on $H^{1/2}(\G)$. The most 
convenient for us is the following Besov-type norm
\begin{equation}\label{eq:Besov}
	\|\psi\|_{1/2}^2 
	=  \|\psi\|_{H^{1/2}(\G)}^2
	:= \|\psi\|_{L^2(\G)}^2 + 
	\int_\G\int_\G \frac{|\psi(x) - \psi(y)|^2}{|x-y|^d}\dd\s(x)\dd\s(y);
\end{equation}
\cf \cite[Eq. (1,3,3,3)]{Gr85}. According to~\cite[Thm. 1.5.1.2]{Gr85} the trace maps are continuous linear operators
\[
	H^1(\Omg\setminus\ov\G)\ni u\mapsto u|_{\G_\pm} \in H^{1/2}(\G).
\]
%

\section{Trace Hardy inequality for a bounded cut}
%
\label{sec:bndcut}
The aim of this section is to prove Theorem~\ref{thm:bnd}.
To this end we need an auxiliary lemma.

\begin{lem}\label{ine}
	Let $\Sg\subset\dR^d$ be a bounded Lipschitz hypersurface 
	satisfying Hypothesis \ref{hyp:1}.
	Let $\Omg\subset\dR^d$ be a bounded Lipschitz  
   domain such that $\ov{\Sg}\subset\Omg$.
	 Then the following statements hold.
	\begin{myenum}
		\item For any $\eps > 0$, there exists a constant $C_\eps = C_\eps(\Sg)>0$ such that
		\begin{equation*}
			\big\|[u]_\Sg\big\|^2_{L^2(\Sg)}
			\le
			\eps \|\nb u\|^2_{L^2(\dR^d;\dC^d)}
			+ C_\eps\|u\|^2_{L^2(\dR^d)}
		\end{equation*}
		holds for all $u\in H^1(\dR^d\setminus\ov\Sg)$. 
		\item 
		There exists a constant 
		$\wt C = \wt C(\Omg,\Sg) > 0$ such that
		for any $u \in H^1(\Omg\setminus\ov\Sg)$
		\begin{equation}\label{eq:trace_weight}
			\int_\Sg
			w(x)\big|[u]_\Sg(x)\big|^2\dd\s(x)
			\leq 
			\wt C\|u\|^2_{H^1(\Omg\setminus\ov\Sg)},
			\quad \text{with~~} w(x) := 
			\int_{\G\setminus\ov\Sg} \frac{\dd\s(x)}{|x-y|^d}.
		\end{equation}
		If, in addition, $\G\sm\Sg$ is a $(d-1)$-set in the sense of~\eqref{eq:(d-1)}, then there exists
		a constant $c > 0$ such that $w(x) \ge c\wt\rho_\Sg(x)^{-1} \ge c\rho_\Sg(x)^{-1}$ for all $x\in\Sg$. 
	\end{myenum}
\end{lem}
\begin{proof}
	\noindent {\rm (i)}
	Let $u\in H^1(\dR^d\setminus\ov\Sg)$. 
	Set $\Omg_- := \dR^d\setminus\ov{\Omg_+}$
	and $u_\pm := u|_{\Omg_\pm}$.
	By elementary computations we get
	\begin{equation}\label{eq:elementary}
		\|[u]_\Sg\|^2_{L^2(\Sg)} \le 
		2\|u_+|_{\G}\|^2_{L^2(\G)}  + 2\|u_-|_{\G}\|_{L^2(\G)}^2.
	\end{equation}
	Applying the trace theorem~\cite{M87} (see also \cite[Lem 2.6]{BEL14}),
	we get that for any $\eps > 0$ there exist $C^\pm_\eps > 0$ such that
	\[
		2\|u_\pm|_{\G}\|^2_{L^2(\G)}
		\le 
		\eps \|\nabla u_\pm\|^2_{L^2(\Omg_\pm;\dC^d)} + C^\pm_\eps\|u_\pm\|^2_{L^2(\Omg_\pm)}. 	
	\]
	Summing the above two inequalities for $\Omg_+$ and
	$\Omg_-$, respectively,
	and combining with~\eqref{eq:elementary}
	we get the desired inequality with $C_\eps = \max\{C^+_\eps,C^-_\eps\}$.
	
	\noindent {\rm (ii)}
	Let $u\in H^1(\Omg\setminus\ov\Sg)$. 
	Set $\Omg_- := \Omg\setminus\ov{\Omg_+}$
	and $u_\pm := u|_{\Omg_\pm}$.
	By the Sobolev trace theorem~\cite[Thm 1.5.1.2]{Gr85} we get that there exists
	a constant $C = C(\G) > 0$ such that 
	\[
		\big\|[u]_\G\big\|_{H^{1/2}(\G)}
		\le 
		\big\|u_+|_\G\big\|_{H^{1/2}(\G)}
		+
		\big\|u_-|_\G\big\|_{H^{1/2}(\G)}
		\le C\big\|u\big\|_{H^1(\Omg\setminus\ov\Sg)}.
	\]
	Using the expression~\eqref{eq:Besov}
	for the norm $\|\cdot\|_{H^{1/2}}$
	and Proposition~\ref{prop:Sobolev} we get 
	\[
	\begin{aligned}
		\big\|[u]_\G\big\|_{H^{1/2}(\G)}^2
		& \ge
		\int_\G \int_\G \frac{\big|[u](x) - [u](y)\big|^2}{|x-y|^d}\dd\s(x)\dd\s(y)\\
		& \ge 
		2\int_\Sg
		w(x) 
		|[u](x)|^2\dd\s(x),
	\end{aligned}
	\]%
	where the weight $w\colon \Sg\arr\dR_+$ is
	as in~\eqref{eq:trace_weight}.
	Suppose now that the closed set $\Sg^{\rm c} := \G\sm\Sg$ is a $(d-1)$-set in the sense of~\eqref{eq:(d-1)}.
	Let us fix $x\in\Sg$ and set $r := 2\,{\rm dist}\,(x,\Sg^{\rm c}) =2\wt\rho_\Sg(x)$ where the distance is measured in $\dR^d$. Let $x_\star\in \Sg^{\rm c}$ be such that
	$|x-x_\star| = \frac{r}{2}$. Then we get the following
	lower bound on $w(x)$
	\[
	\begin{aligned}
		w(x)
		& \ge 		
		\int_{\cB_r(x)\cap \Sg^{\rm c}}
		\frac{\dd\s(y)}{|x-y|^d}
		\ge
		\frac{\Lambda_{d-1}(\cB_r(x)\cap \Sg^{\rm c})}{r^d}\\
		& \ge
		\frac{\Lambda_{d-1}(\cB_{\frac{r}{2}}(x_\star)
			\cap\Sg^{\rm c})}{r^d}
		\ge 
		\frac{c_-}{2^{d-1}r} =\frac{c}{\wt\rho_\Sg(x)} \ge
		\frac{c}{\rho_\Sg(x)};
	\end{aligned}
	\]
	where $c := \frac{c_-}{2^d}$. 
	In the course of these estimates we used that
	$\cB_{\frac{r}{2}}(x_\star)\subset\cB_r(x)$
	and that $\rho_\Sg(x) \ge \frac{r}{2}$.
\end{proof}	
Now we have all the tools to prove the desired functional inequality. For the sake of convenience 
we introduce the notation
\begin{equation}\label{eq:jumpform}
	\frh_\Sg(u) := \int_\Sg 
	w(x)|[u]_\Sg(x)|^2\dd \s(x).
\end{equation}
\begin{proof}[Proof of Theorem~\ref{thm:bnd}]
	Recall that $\Sg$ is a bounded  non-closed
	Lipschitz hypersurface as in Hypothesis~\ref{hyp:1}. 
	Let $\Omg\subset\dR^d$ be a bounded
	connected $\cC^\infty$-smooth domain
	such that the inclusion $\ov{\Sg}\subset\Omg$ holds and that $\Omg\setminus\ov{\Sg}$ is connected as well.
	For any $u \in H^1(\dR^d\setminus\ov\Sg)$
	we clearly have $\fru := u|_{\Omg}\in H^1(\Omg\setminus\ov\Sg)$ and moreover
	\begin{equation}\label{eq:est_grad1}
		\int_{\dR^d} |\nb u|^2\dd x
		\ge	
		\int_\Omg | \nabla \fru|^2\dd x.
	\end{equation}
	Since the domain $\Omg$ has a finite measure,
	Cauchy-Schwarz inequality implies
	$\fru\in L^1(\Omg)$. 
	Thus, the average 
	\[
		\avg{\fru} = \frac{1}{|\Omg|}\int_\Omg \fru(x) \dd x
	\]
	of $\fru$ is well defined.
	Rather elementary calculations give
	\begin{equation*}
		\left[\fru - \avg{\fru}\right]_\Sg = [\fru]_\Sg
		\and
		\nabla\big(\fru-\avg{\fru}\big)
		=
		\nabla\fru.
	\end{equation*}
	Note that the constant function on $\Omg$ is the eigenfunction corresponding to the lowest eigenvalue $\lm_1^{\rm N}(\Omg\setminus\ov\Sg) = 0$  of the Neumann Laplacian on $\Omg\setminus\ov\Sg$ and that $\fru -\avg{\fru}$ is orthogonal to it.
	Hence, by the min-max principle we can estimate the $L^2$-norm of the difference $\fru - \avg{\fru}$ in the following way
	\begin{equation*}
	\lm_2^{\rm N}(\Omg\setminus\ov\Sg)
	\|\fru- \avg{\fru}\|^2_{L^2(\Omg)}
	\leq\|\nabla(\fru-\avg{\fru} )\|^2_{L^2(\Omg;\dC^d)},
	\end{equation*}
	where $\lm_2^{\rm N}(\Omg\sm
	\ov\Sg)$ denotes the second eigenvalue of the Neumann Laplacian on the domain $\Omg\setminus\ov\Sg$. 
	Observe that $\lm_2^{\rm N}(\Omg\setminus\ov\Sg) > 0$, because
	the domain $\Omg\setminus\ov{\Sg}$ is connected and bounded.  
	
	Using the shorthand notation $\kp :=
	(\lm_2^{\rm N}(\Omg\setminus\ov\Sg))^{-1}$, we rewrite the inequality in Lem\-ma~\ref{ine}\,(ii) for the function $\fru-\avg{\fru}$ as follows
	\begin{equation*}
	\label{ner}
	\begin{split}
		\frh_\Sg[u]
		&=
		\int_\Sg w(x)|[\fru]_\Sg(x)|^2 \dd\s(x)=
		\int_\Sg w(x)|[\fru - \avg{\fru}]_\Sg(x)|^2 \dd\s(x)\\[0.4ex]
		& 
		\leq 
		\wt C
		\|\fru -\avg{\fru}\|_{H^1(\Omg\setminus\ov\Sg)}^2
		\leq 
		\wt C(1+\kp)
		\|\nabla(\fru-\avg{\fru})\|^2_{L^2(\Omg;\dC^d)}
		\\[0.4ex]
		&=
		\wt C\left(1+\kp\right)\|\nabla \fru\|^2_{L^2(\Omg;\dC^d)}
		 \le
		\wt C\left(1+\kp\right)\|\nb u\|^2_{L^2(\dR^d;\dC^d)}.
	\end{split}
	\end{equation*}
	Thus, we get the desired inequality
	with $C^{-1} = \wt C(1+\kp)$.
	
	If, in addition, the $\G\sm\Sg$ is a $(d-1)$-set in the sense of~\eqref{eq:(d-1)}, we get the inequality
	\[	
		\int_{\dR^d}|\nb u|^2\dd x  \ge
			C\int_\Sg\frac{|[u]_\Sg|^2}{\wt\rho_\Sg(x)}\dd x
		\ge
		C\int_\Sg\frac{|[u]_\Sg|^2}{\rho_\Sg(x)}\dd x.
		\qedhere
	\]
\end{proof}


\section{Applications}
\label{appl}
In this section we discuss applications
of the obtained inequality. 
First, in Subsection~\ref{ssec:heat} we apply our main result to the heat
equation. Second, in Subsection~\ref{ssec:delta}
we apply our main result to the Schr\"odinger operator
with a $\dl'$-interaction supported on a non-closed hypersurface.

\subsection{Applications to the propagation of heat}
\label{ssec:heat}

The trace Hardy inequality in Theorem~\ref{thm:bnd} finds a physically motivated 
application in the theory of the heat propagation. 
%


Let $\Sg\subset\dR^d$ be a bounded non-closed Lipschitz
hypersurface as in Section~\ref{sec:mainres}.
Consider the closed, non-negative, and densely defined sesquilinear form in
the Hilbert space $L^2(\dR^d)$

\begin{equation}\label{eq:form}
	\fra_\Sg(u,v) := \left(\nb u,\nb v\right)_{L^2(\dR^d;\dC^d)},
	\qquad 
	\dom \fra_\Sg := H^1(\dR^d\setminus\ov\Sg).
\end{equation}

The above sesquilinear form defines
by the first representation theorem a unique self-adjoint operator $\sfA_\Sg$ in the Hilbert space $L^2(\dR^d)$. This operator corresponds to the (minus) Laplace operator with Neumann boundary conditions on the cut $\Sg$.
In thermodynamics, the Neumann boundary condition
describes the heat insulator; 
\ie roughly speaking, a barrier through which the heat does not propagate. By~\cite[Thm. 1.49]{Ouh05}, the operator $\sfA_\Sg$  generates a strongly continuous contraction semigroup $e^{-t\sfA_\Sg}$ on $L^2(\dR^d)$. This semigroup can be defined  \eg through the functional calculus for self-adjoint operators. 
Our main interest is the large-time behaviour 
of the function 
\[
	t\mapsto\frh_\Sg\left(e^{-t\sfA_\Sg} u\right)
	= \int_{\Sg} w(x)[e^{-t\sfA_\Sg}u]_\Sg(x)\dd \s(s),
\]	
where $w$ is as in~\eqref{eq:HardyWeight}.
This function can be used to measure the jump of the temperature
across the insulating hypersurface $\Sg$ at the time
$t> 0$.
Employing the Hardy inequality in Theorem~\ref{thm:bnd}
we get an integral estimate on $\frh_\Sg\left(u(\cdot,t)\right)$. 
\begin{thm}\label{prop:heat_bnd}
	Let $\Sg\subset\dR^d$ be a bounded  Lipschitz 
	hypersurface as in Section~\ref{sec:mainres}. Let the self-adjoint operator $\sfA_\Sg$ be associated with the form~\eqref{eq:form}. Then there exists a constant $C = C(\Sg) > 0$
	such that
	\[
		\int_0^\infty\frh_\Sg(e^{-t\sfA_\Sg} u) \dd t \le 
		C
		\|u\|^2_{L^2(\dR^d)}
	\] 
	holds for all $u \in H^1(\dR^d\setminus\ov\Sg)$.
\end{thm}
%
\begin{proof}
	We denote by $\cV'$ the dual space to the Sobolev space $\cV := H^1(\dR^d\setminus\ov\Sg)$ and by $\langle\cdot,\cdot\rangle$ the corresponding duality product, which is compatible with the inner product $(\cdot,\cdot)_{L^2(\dR^d)}$. Clearly, we have the chain of inclusions
	\[
	\cV\subset L^2(\dR^d)\subset \cV'.
	\]
	The form $\fra_\Sg$ defines (see~\cite[Section 1.4.2]{Ouh05}) a unique linear operator
	$\cA_\Sigma$ in the Banach space $\cV'$ with $\dom\cA_\Sg = \cV$ such that
	\begin{equation}\label{eq:cA_Sg}
		\fra_\Sg(u,v) = \langle\cA_\Sg u, v \rangle,
		\qquad\forall\, u,v\in\dom\fra_\Sg.
	\end{equation}
   By~\cite[Thm. 1.55]{Ouh05} $\cA_\Sg$
	generates a strongly continuous  semigroup $e^{-t\cA_\Sg}$ on $\cV'$, which is compatible with the semigroup $e^{-t\sfA_\Sg}$ in the following sense 
	\begin{equation}\label{eq:compatibility}
	e^{-t\cA_\Sg}u = e^{-t\sfA_\Sg}u,\qquad \forall\, u\in L^2(\dR^d).
	\end{equation}
	According to~\cite[Lem. 6.1.11]{Da07} we have $e^{-t\cA_\Sg}u\in \cV$ for any $u\in \cV$. Moreover, by~\cite[Lem. 6.1.13]{Da07} the function $t\mapsto e^{-t\cA_\Sg}u$  is 
	norm continuously differentiable on $[0,\infty)$ and
	\begin{equation}\label{eq:derivative}
	\frac{\dd}{\dd t} e^{-t\cA_\Sg}u = 
	-\cA_\Sg e^{-t\cA_\Sg}u.
	\end{equation}	
	For any $u\in\cV$ we get using~\eqref{eq:cA_Sg},~\eqref{eq:compatibility} and~\eqref{eq:derivative} 
	\[
	\begin{aligned}
	\frac{\dd}{\dd t}\|e^{-t\sfA_\Sg}u\|^2_{L^2(\dR^d)}& =
	\frac{\dd}{\dd t}( e^{-2t\sfA_\Sg}u,u)_{L^2(\dR^d)}
	=
	\frac{\dd}{\dd t}\langle e^{-2t\cA_\Sg}u,u\rangle
	=
	-2\langle \cA_\Sg e^{-2t\cA_\Sg}u,u\rangle\\
	& =
	-2\fra_\Sg(e^{-2t\sfA_\Sg}u,u).
	\end{aligned}
\]
Applying the second representation theorem~\cite[Thm. VI 2.23]{Kato}
\[	
\begin{aligned}
	\frac{\dd}{\dd t}\|e^{-t\sfA_\Sg}u\|^2_{L^2(\dR^d)}
	& =-2(\sfA_\Sg^{1/2}e^{-2t\sfA_\Sg},\sfA_\Sg^{1/2} u)_{L^2(\dR^d)} \\
	& = 
	-2(\sfA_\Sg^{1/2}e^{-t\sfA_\Sg}u,\sfA_\Sg^{1/2}e^{-t\sfA_\Sg} u)_{L^2(\dR^d)}
	= -2\fra_\Sg(e^{-t\sfA_\Sg}u).
\end{aligned}\]
Integrating against $t$ the above equation on the interval $(0,\infty)$ and using the inequality in Theorem~\ref{thm:bnd} and the fact that $e^{-t\sfA_\Sg}$ is a contraction semigroup we get
that there exists a constant $C > 0$ such that
\[
	C\int_0^\infty\frh_\Sg(e^{-t\sfA_\Sg} u)
	\le -\int_0^\infty \frac{\dd }{\dd t}\|e^{-t\sfA_\Sg} u\|^2_{L^2(\dR^d)}\dd t
	\le \|u\|^2_{L^2(\dR^d)},
\] 
by which the theorem is proved.
\end{proof}


\subsection{Application to $\delta'$-interacion}
\label{ssec:delta}
Using Theorem~\ref{thm:bnd} we can show the absence of negative discrete spectrum for Hamiltonians describing weak $\delta'$-interactions supported on non-closed hypersurfaces. 

First, we briefly recall how these operators are defined. For a more complete discussion we refer reader to~\cite{BEL14, BLL13AHP} for the case of closed hypersurfaces or to \cite{ER16, JL16} for the non-closed case. Let $\Omega_+$ be a bounded Lipschitz domain as in
Section~\ref{sec:mainres} such that $\Sigma$ is a relatively open connected subset of its boundary $\G := \p\Omg_+$. We introduce the quadratic form
\begin{equation}\label{eq:deltacl}
	\fra_{\omg,\G}(u)
	 =
	\left\|
		\nabla u
	\right\|^2_{L^2(\dR^d;\dC^d)}
	-
	\left(\omg [u]_{\Sg}, [u]_{\Sg}\right)_{L^2(\Sg)},\qquad
	\dom\fra_{\omg,\G}  := H^1(\dR^d\setminus\ov\G),
\end{equation}
where $\omg\in L^\infty(\Sg;\dR)$ denotes the coupling coefficient. The above form is
closed, densely defined, symmetric and lower semi-bounded; \cf \cite[Prop. 3.1]{BEL14}. As the next step, we introduce   a restriction of the form \eqref{eq:deltacl}
\begin{equation*}
	\fra_{\omg,\Sg}(u)
	=
	\fra_{\omg,\G}(u),
	\qquad \dom\fra_{\omg,\Sg}:=\{u\in \dom\fra_{\omg,\Sg}\colon [u]_{\G\setminus\ov\Sg}=0\}.
\end{equation*}
This form is obviously closed, densely defined, symmetric and lower semi-bounded. The operator $\sfA_{\omg,\Sg}$ describing $\delta'$-interaction is defined as the unique self-adjoint operator representing the form $\fra_{\omg,\Sg}$ in the usual manner. Now we are ready to state the main
result on $\dl'$-interactions. 

\begin{prop}\label{prop:delta}
	Let $\Sg\subset\dR^d$ be a bounded non-closed Lipschitz 
	hypersurface satisfying Hypothesis \ref{hyp:1}.
	Then there is a constant $\omg_{\star}  =\omg_{\star}(\Sg) > 0$ such that 
	$\sfA_{\omg,\Sg}\ge 0$
	holds if $\omg(x)\le\omg_{\star}$ for all $x\in\Sg$.
\end{prop}
\begin{proof}
Applying Theorem~\ref{thm:bnd} we find that
\begin{equation}
\label{proof_delta'}
	\fra_{\omg,\Sg}(u)=	\left\|
		\nabla u
	\right\|^2_{L^2(\dR^d;\dC^d)}
	-
	\int_\Sg\omega |[u]_\Sg|^2\dd \s(x)  \ge
	\int_\Sg (w-\omega) |[u]_\Sg|^2\dd \s(x),
\end{equation}
where $w$ is as in Theorem~\ref{thm:bnd}. We choose $\omg_{ \star} :=\inf_{x\in\Sg}w(x) > 0$. The right hand side of \eqref{proof_delta'} is certainly non-negative provided $\omg_{\star}\geq\omega$ on $\Sg$,  which completes the proof.
\end{proof}	
The above proposition yields the following elementary consequence.
\begin{cor}
	Let $\Sg\subset\dR^d$ be a bounded non-closed Lipschitz 
	hypersurface satisfying Hypothesis \ref{hyp:1}.
	Let the coupling coefficient be constant $\omg\in\dR$.
	Then there is a critical coupling coefficient $\omg_{\rm c}  =\omg_{\rm c}(\Sg) > 0$ such that 
	$\sfA_{\omg,\Sg}\ge 0$
	holds if and only if $\omg\in(-\infty,\omg_{\rm c}]$.
\end{cor}
\begin{remark}
	If we consider a family of nested Lipschitz hypersurfaces $(\Sg_\eps)_{\eps > 0}$ in $\G$ with $\Sg_{\eps'} \supset\Sg_\eps$ for $\eps' < \eps$  such that $\G\setminus\ov\Sg_\eps$ shrinks to a point $x_0$  in $\G$, then we expect that the critical coupling coefficient $\omg_{\rm c}(\Sg_\eps)\arr 0$ as $\eps \arr 0$. 
If the family of hypersurfaces $(\Sg_\eps)_{\eps > 0}$ enjoys some self-similarity properties around $x_0$, 
asymptotic expansion of $\omg_{\rm c}(\Sg_\eps)$ in the small parameter $\eps >0$ can be determined  using the technique of~\cite{BFTT12, DTV10}.
\end{remark}

\section{Generalizations to unbounded and to non-orientable cuts}\label{sec:unb}
So far, we have considered only the situation when the cut is bounded and orientable. In this section we will discuss a class of unbounded cuts. The same technique is applicable for cuts given by non-orientable hypersurfaces.

As a warm-up we will derive from~\eqref{eq:Hardy_trace} a simple trace Hardy inequality for a not necessarily bounded cut contained in a hyperplane.
On the Euclidean space $\dR^d$, $d \ge 3$ we define the coordinates $x = (x',x_d)^\top$ with $x'\in\dR^{d-1}$ and $x_d\in\dR$. We also define the upper and lower half-spaces 
\[
	\dR^d_\pm := \{(x',x_d)\colon x'\in\dR^{d-1}, \pm x_d > 0\}.
\]
The common boundary of $\dR^d_+$ and $\dR^d_-$ is a hyperplane
$\G := \dR^{d-1}\times \{0\}$. Let $\Sg$ be a relatively open subset of $\G$, which is not assumed to be bounded. The Sobolev space $H^1(\dR^d\sm\ov\Sg)$ can be defined in a similar way as in Subsection~\ref{sec:Sobolev}
\[
	H^1(\dR^d\sm\ov\Sg) = \big\{u = u_+\oplus u_- \in H^1(\dR^d_+)\oplus H^1(\dR^d_-)\colon
	u_+|_{\G\sm\ov\Sg} = u_-|_{\G\sm\ov\Sg}\big\}.
\]
For any $u\in H^1(\dR^d\sm\ov\Sg)$, we find using the inequality~\eqref{eq:Hardy_trace} that
\begin{align}
	\int_{\dR^d}|\nabla u|^2\dd x&  = 
	\int_{\dR^d_+}|\nabla u|^2\dd x  +
	\int_{\dR^d_-}|\nabla u|^2\dd x \nonumber\\
	& \ge 2\left(\frac{\G(\frac{d}{4})}{\G(\frac{d-2}{4})}\right)^2\int_\G\left(
	\frac{|u_+|_{\G}|^2}{|x'|} + \frac{|u_-|_{\G}|^2}{|x'|}\right)\dd x' \nonumber\\
	\label{eq:Rd-1}
	& \ge
	\left(\frac{\G(\frac{d}{4})}{\G(\frac{d-2}{4})}\right)^2\int_\Sg
	\frac{\big|u_+|_{\Sg} - u_-|_{\Sg}\big|^2}{|x'|}\dd x'.
\end{align}
Modifying the technique of the proof of Theorem~\ref{thm:bnd} we obtain
for a class of unbounded cuts an improvement  of the above trace Hardy inequality with the weight which is singular in the neighbourhood of $\p\Sg$ and behaves at infinity far from $\p\Sg$ as $|x'|^{-1}$. 
Recall that $\cS$ denotes the unit sphere of $\dR^d$ and also that $\G := \dR^{d-1}\times \{0\}$, $d\ge 3$, and set $\hat\Gamma:= \Gamma\cap\cS$.
Let $\hat\Sigma$ be an open set in $\hat\Gamma$ such that $\hat\Gamma\setminus\hat\Sigma$ is a $(d-2)$-set.
Define the cone $\Sigma$ as
\begin{equation}\label{eq:cone}
   \Sigma = \left\{x\in\dR^d\setminus\{0\}\vl{\colon} \frac{x}{|x|} \in \hat\Sigma\right\}.
\end{equation}
Let $\rho_\Sigma$ be the geodesic distance in $\Sg$ to $\partial\Sigma$. We also remark that it coincides with a counterpart Euclidean distance $\wt\rho_\Sg$.

\begin{thm}
\label{thm:cone}
Let the cone $\Sg\subset\dR^d$ be as in~\eqref{eq:cone}.
Then there exists a constant $C = C(\Sg) > 0$ such that 
	\[
	\int_{\dR^d}|\nb u(x)|^2\dd x 
	\ge 
	C\int_\Sg
	\rho_\Sg(x')^{-1} |[u]_\Sg(x')|^2\dd x',
	\qquad 
	\forall\, u\in H^1(\dR^d\setminus\ov\Sg).
	\]
\end{thm}

\begin{remark}
Denote by $\wt\rho_{\hat\Sigma}$ the (Euclidean) distance to $\partial\hat\Sigma$ and by $\hat x$ the ``angle'' $\frac{x}{|x|}$ for any $x\in\dR^d\setminus\{0\}$. There holds
\begin{equation}
\label{eq:dist}
   \tfrac{1}{2}\,|x'| \,\wt\rho_{\hat\Sigma}(\hat x') \le \rho_\Sigma(x') \le |x'| \,\wt\rho_{\hat\Sigma}(\hat x'),
   \quad x'\in\Sigma\setminus\{0\}.
\end{equation} 
\end{remark}


\begin{proof}
Consider the following ``partition'' of $\dR^d$ by dyadic spherical shells: For any integer $j\in\dZ$ let $\cA_j$ be the spherical shell
\[
   \cA_j = \cB_{2^{j+1}}(0) \setminus \ov{\cB_{2^j}(0)} = \big\{x\in\dR^d\colon |x|\in(2^j,2^{j+1})\big\}.
\] 

Then for any function $f\in L^2(\dR^d)$ there holds
\[
   \int_{\dR^d} |f(x)|^2\dd x = \sum_{j\in\dZ} \int_{\cA_j} |f(x)|^2\dd x\,.
\]
Set $\Sigma_j=\cA_j\cap\Sigma$. Then the theorem will be proved if we show that 
\begin{equation}
\label{eq:Aj}
	\int_{\cA_j}|\nb u(x)|^2\dd x 
	\ge 
	C\int_{\Sigma_j}
	\rho_\Sg(x')^{-1} |[u]_{\Sigma_j}(x')|^2\dd x',
	\quad 
	\forall\, u\in H^1(\cA_j\setminus\ov\Sigma_j),
\end{equation}
with a constant $C$ independent of $j$. Using the distance equivalence \eqref{eq:dist}, we see that \eqref{eq:Aj} follows from
\begin{equation}
\label{eq:Ajbis}
	\int_{\cA_j}|\nb u(x)|^2\dd x 
	\ge 
	2C\int_{\Sigma_j}
	|x'|^{-1}\wt\rho_{\hat\Sigma}(\hat x')^{-1} |[u]_{\Sigma_j}(x')|^2\dd x',
	\quad 
	\forall\, u\in H^1(\cA_j\setminus\ov\Sigma_j).
\end{equation}
Now, the norms in both members of \eqref{eq:Ajbis} are homogeneous by dilatation and have the same homogeneity. As a consequence if we prove \eqref{eq:Ajbis} for $j=0$ with a constant $C=C_0$, then we get \eqref{eq:Ajbis} for any $j\in\dZ$ with the same constant $C=C_0$ by virtue of the change of variables $x\mapsto 2^jx$ from $\cA_0$ to $\cA_j$. Using that $\wt\rho_{\Sg_0}(x') \le |x'|\wt\rho_{\hat\Sg}(\hat{x}')$, we see that the estimate \eqref{eq:Ajbis} for $j=0$ follows from
\begin{equation}
\label{eq:A0}
	\int_{\cA_0}|\nb u(x)|^2\dd x 
	\ge 
	2C_0\int_{\Sigma_0}
	\wt\rho_{\Sigma_0}(x')^{-1} |[u]_{\Sigma_0}(x')|^2\dd x',
	\quad 
	\forall\, u\in H^1(\cA_0\setminus\ov\Sigma_0)
\end{equation}
for some positive constant $C_0$. Now we are (almost) back to what we have proved in Section \ref{sec:bndcut}. Let $\Omega_\pm = \cA_0\cap\dR^d_\pm$. We are in a bounded configuration with the cut $\Sigma_0$ contained in the flat surface $\Gamma_0=\Gamma\cap\cA_0$. Since $\hat\Gamma\setminus\hat\Sigma$ is a $(d-2)$-set, we find that $\Gamma_0\setminus\Sigma_0$ is a $(d-1)$-set and that $\Omega:=\cA_0\setminus\ov{\Sigma_0}$ is a connected set. Both configurations (the present one and that in Section \ref{sec:bndcut}) differ by the fact that $\Gamma_0$ is not the full boundary of $\Omega_\pm$ any more. Nevertheless, we can check that all steps of the former proof are valid: characterization of traces from $H^1(\Omega_\pm)$ to $H^{1/2}(\Gamma_0)$ and positiveness of the second Neumann eigenvalue in $\Omega$. This yields the proof of \eqref{eq:A0}, and hence the proof of the theorem.
\end{proof}

\begin{remark}
Theorem \ref{thm:cone} can be generalized in the following way: Let $\hat\Gamma$ be a closed Lipschitz hypersurface in $\cS$, such that $\cS\setminus\hat\Gamma$ has two connected components $\hat\Omega_\pm$ that are Lipschitz subdomains of $\cS$. Set $\Gamma$ as the closed Lipschitz hypersurface of $\dR^d$ defined by
\[
   \Gamma = \Big\{x\in\dR^d\setminus\{0\},\quad \frac{x}{|x|} \in \hat\Gamma\Big\}.
\]
With these assumptions on $\Gamma$ replacing that $\Gamma$ is an hyperplane, Theorem \ref{thm:cone} still holds.
\end{remark}

The proof of Theorem \ref{thm:cone} derives from Theorem \ref{thm:bnd} via a localization process. The localized estimates are not a direct consequence of Theorem \ref{thm:bnd}, but are obtained by the same chain of arguments. The assumptions of Hypothesis \ref{hyp:1} are indeed sufficient, but not necessary. Let us give two examples.

\begin{example}[The penny shape hole]
Let $\Gamma$ be the hyperplane $\dR^{d-1}\times\{0\}$ and $\Sigma = \Gamma\setminus\ov{\cB_1(0)}$. This configuration is smooth, but satisfies neither the assumptions of Theorem \ref{thm:bnd} nor those of Theorem \ref{thm:cone}. But if we localize to $\Omega=\cB_2(0)$, we obtain, by the same process as before the local estimate (with $\Sigma_0=\Sigma\cap\Omega$, but $\rho_\Sigma$ still the distance to $\partial\Sigma$)
	\[
	\int_{\Omega}|\nb u(x)|^2\dd x 
	\ge 
	C\int_{\Sigma_0}
	\rho_{\Sigma}(x')^{-1} |[u]_{\Sigma_0}(x')|^2\dd x',
	\qquad 
	\forall\, u\in H^1(\Omega\setminus\ov{\Sigma_0}).
	\]
We note that outside $\Omega$, the function $x'\mapsto\rho_{\Sigma}(x')$ is equivalent to $|x'|$. Hence, combining the latter inequality with the estimate \eqref{eq:Rd-1}, we obtain
	\[
	\int_{\dR^d}|\nb u(x)|^2\dd x 
	\ge 
	C\int_{\Sigma}
	\rho_{\Sigma}(x')^{-1} |[u]_{\Sigma}(x')|^2\dd x',
	\qquad 
	\forall\, u\in H^1(\dR^d\setminus\ov{\Sigma}).
	\]
\end{example}

\begin{example}[The M\"obius strip]\label{ex:moebius}
Consider (in $\dR^3$) a smooth M\"obius strip $\Sigma$ surrounded by a torus $\Omega$. More specifically, suppose that $\Sigma$ is generated by smoothly twisting a segment of length $1$ along a circular curve $\gamma$ of radius $2$ (for any $x\in\gamma$, the intersection of $\Sigma$ with the plane $\Pi_x$ orthogonal to $\gamma$ at $x$ is a segment $I_x$ such that $x$ is the middle of $I_x$) and $\Omega$ is the torus of major radius $2$ and minor radius $1$. Choose as $\Gamma$ the larger M\"obius strip generated by the same curve $\gamma$ and intervals $J_x\supset I_x$ of length $2$ and center $x$. Signed jump of a function cannot be defined across such a $\Sigma$ but the jump norm can. The torus $\Omega$ can be cut into several slices so that the topology of each corresponding pieces of $\Sigma$ and $\Gamma$ is trivial. Adding the finitely many contributions one obtains
	\[
	\int_{\Omega}|\nb u(x)|^2\dd x 
	\ge 
	C\int_{\Sigma}
	\rho_{\Sigma}(x)^{-1} |[u]_{\Sigma}(x)|^2\dd\sigma(x) ,
	\qquad 
	\forall\, u\in H^1(\Omega\setminus\ov{\Sigma}).
	\]
\end{example}

\newcommand{\etalchar}[1]{$^{#1}$}


\end{document}